\documentclass[12pt]{amsart}
\usepackage{amsmath,amssymb,amsbsy,amsfonts,amsthm,latexsym,mathabx,
            amsopn,amstext,amsxtra,euscript,amscd,stmaryrd,mathrsfs,
            cite,array,mathtools,enumerate}

\usepackage{url}
\usepackage[colorlinks,linkcolor=blue,anchorcolor=blue,citecolor=blue,backref=page]{hyperref}

\usepackage{color}

\usepackage{color}
\usepackage{float}

\hypersetup{breaklinks=true}

\usepackage[english]{babel}

\usepackage{mathtools}
\usepackage{todonotes}

\usepackage[norefs,nocites]{refcheck}

\usepackage{enumitem}

\def\sC {\mathscr{C}}
\def\sP {\mathscr{P}}

\usepackage{todonotes}
\usepackage{url}
\usepackage[colorlinks,linkcolor=blue,anchorcolor=blue,citecolor=blue,backref=page]{hyperref}
\begin{document}

\newtheorem{theorem}{Theorem}
\newtheorem{lemma}[theorem]{Lemma}
\newtheorem{claim}[theorem]{Claim}
\newtheorem{cor}[theorem]{Corollary}
\newtheorem{prop}[theorem]{Proposition}
\newtheorem{definition}{Definition}
\newtheorem{question}[theorem]{Open Question}
\newtheorem{example}[theorem]{Example}

\numberwithin{equation}{section}
\numberwithin{theorem}{section}

 \newcommand{\F}{\mathbb{F}}
\newcommand{\K}{\mathbb{K}}
\newcommand{\D}[1]{D\(#1\)}
\def\scr{\scriptstyle}
\def\\{\cr}
\def\({\left(}
\def\){\right)}
\def\[{\left[}
\def\]{\right]}
\def\<{\langle}
\def\>{\rangle}
\def\fl#1{\left\lfloor#1\right\rfloor}
\def\rf#1{\left\lceil#1\right\rceil}
\def\le{\leqslant}
\def\ge{\geqslant}
\def\eps{\varepsilon}
\def\mand{\qquad\mbox{and}\qquad}

\def\Res{\mathrm{Res}}
\def\vec#1{\mathbf{#1}}

\def\bl#1{\begin{color}{blue}#1\end{color}} 

\newcommand{\Fq}{\mathbb{F}_q}
\newcommand{\Fp}{\mathbb{F}_p}
\newcommand{\Disc}[1]{\mathrm{Disc}\(#1\)}

\newcommand{\Z}{\mathbb{Z}}
\newcommand{\C}{\mathbb{C}}
\newcommand{\Q}{\mathbb{Q}}
\renewcommand{\L}{\mathbb{L}}

\def\cA{{\mathcal A}}
\def\cB{{\mathcal B}}
\def\cC{{\mathcal C}}
\def\cD{{\mathcal D}}
\def\cE{{\mathcal E}}
\def\cF{{\mathcal F}}
\def\cG{{\mathcal G}}
\def\cH{{\mathcal H}}
\def\cI{{\mathcal I}}
\def\cJ{{\mathcal J}}
\def\cK{{\mathcal K}}
\def\cL{{\mathcal L}}
\def\cM{{\mathcal M}}
\def\cN{{\mathcal N}}
\def\cO{{\mathcal O}}
\def\cP{{\mathcal P}}
\def\cQ{{\mathcal Q}}
\def\cR{{\mathcal R}}
\def\cS{{\mathcal S}}
\def\cT{{\mathcal T}}
\def\cU{{\mathcal U}}
\def\cV{{\mathcal V}}
\def\cW{{\mathcal W}}
\def\cX{{\mathcal X}}
\def\cY{{\mathcal Y}}
\def\cZ{{\mathcal Z}}

\def \brho{\boldsymbol{\rho}}

\def \pf {\mathfrak p}

\def \Prob{{\mathrm {}}}
\def\e{\mathbf{e}}
\def\ep{{\mathbf{\,e}}_p}
\def\epp{{\mathbf{\,e}}_{p^2}}
\def\em{{\mathbf{\,e}}_m}

\newcommand{\sR}{\ensuremath{\mathscr{R}}}
\newcommand{\sDI}{\ensuremath{\mathscr{DI}}}
\newcommand{\DI}{\ensuremath{\mathrm{DI}}}

\newcommand{\Nm}[1]{\mathrm{Norm}_{\,\F_{q^k}/\Fq}(#1)}

\def\Tr{\mbox{Tr}}
\newcommand{\rad}[1]{\mathrm{rad}(#1)}

\newcommand{\Orb}[1]{\mathrm{Orb}\(#1\)}
\newcommand{\aOrb}[1]{\overline{\mathrm{Orb}}\(#1\)}

%

%

\title[Orbits in Small Subgroups]
{Unlikely Intersections over Finite 
Fields: Polynomial Orbits in Small Subgroups}

  \author[L. M{\'e}rai]{L{\'a}szl{\'o} M{\'e}rai}
 \address{L.M.: Johann Radon Institute for Computational and Applied Mathematics, Austrian Academy of Sciences,  Altenberger Stra\ss e 69, A-4040 Linz, Austria} 
 \email{laszlo.merai@oeaw.ac.at}

 \author[I.~E.~Shparlinski]{Igor E. Shparlinski}
 \address{I.E.S.: School of Mathematics and Statistics, University of New South Wales.
 Sydney, NSW 2052, Australia}
 \email{igor.shparlinski@unsw.edu.au}

\pagenumbering{arabic}

\begin{abstract}
We  estimate the  frequency of polynomial iterations which fall in a given multiplicative subgroup 
of a finite field of $p$ elements. We also  give a lower bound on the size of the subgroup which is multiplicatively 
generated by the first $N$ elements in an orbit. We derive these  from more general results about 
sequences of compositions on a fixed set of polynomials. 
\end{abstract}


\maketitle

\section{Introduction}

\subsection{Background and motivation} 

Recently, several of so called  {\it unlikely intersection\/} type  results, see~\cite{Zan} for a general background, 
have been obtained on the scarcity of 
elements in orbits of polynomial maps
in fields  of characteristic zero that  fall in a set  of prescribed additive, 
multiplicative  or algebraic structure.  Examples of such  sets include \label{sec:motiv} \begin{itemize}  
\item algebraic varieties~\cite{BGKT,OstSha,SiVi,Xie1,Xie2} where the problem is also known as the
{\it dynamical Mordell-Lang conjecture\/}; 
\item an  orbit  generated by another polynomial or rational function~\cite{GTZ1,GTZ2};
\item  the set of all roots of unity in $\C$, see~\cite{DZ,Ost1} and more generally, of algebraic numbers
with all conjugates  bounded by some constant, see~\cite{Chen,OY};
\item   the set  of all perfect powers in a number field, see~\cite{CJS,OPS};
\item  a finitely generated group in a number field, see~\cite{BOSS,KLSTYZ,OstShp}. 
\end{itemize}

As we have mentioned, the above results are all established  for  fields of characteristic zero. 
There are also some analogues, 
or sometimes even stronger results
for  function  fields of positive characteristic, where some additional tools are available, such as a very strong (and rigorously established) 
form of the $abc$-conjecture (see~\cite{Mas}).

On the other hand, there are very few results in this direction in the settings of finite fields. In fact, the problem itself 
has to be reformulated and instead of asking for finiteness  results (which is immediate 
in this case) one can ask about upper bounds on the size of these intersections compared to the orbit length. 
For example, multiplicative character sums along elements of very long orbits have been estimated in~\cite{NiWi1}
(using the ideas of~\cite{NiWi2} one can apparently improve the result of~\cite{NiWi1} and make it nontrivial for slightly shorter 
orbits).  
In turn, this immediately allows to study intersections of orbits with multiplicative subgroups of finite fields. 
Some of the results of~\cite{Shp1, Shp2} apply to very short segments of an orbit,  however they are not uniform with 
respect to the polynomials $f$ defining the dynamical system, that is, the bounds depend on the size of the 
coefficients of $f$, which is not quite natural in the settings of finite fields. It is interesting to note that relatively more is 
known about {\it additive properties\/}  of orbits in finite fields, where a range of new techniques becomes available, 
see~\cite{CGOS, Chang1, Chang2, Ost2, RNS}. 

The main goal of this  work is to bridge this gap and present some results which are  both nontrivial to short segments of orbits 
and uniform with respect to the polynomial  $f$.  

Additionally we consider a more general case, when instead of one polynomial,
 finitely many polynomials are composed with each other. In other words, we consider orbits of a 
semigroup generated by a finite set of  polynomials over a finite field. 

Our technique is based on a very recent result of Vyugin~\cite{Vyu} which is based on a rather delicate blend of techniques 
and ideas coming from additive combinatorics and the so-called {\it Stepanov method\/}, which is of an algebraic nature. 

\subsection{Set-up} 
Here, motivated by the  results outline in Section~\ref{sec:motiv}, we consider analogous problems in positive characteristic and in particular 
investigate the frequency of values in an orbit of a polynomial over a finite field which fall in a  
multiplicative subgroup of a given order.  

Let $\F_q$ be the finite field of $q$ elements. Given a polynomial $f\in\F_q[X]$, we consider the trajectories generated by iterations of $f$ starting from some $u\in\F_q$, that is, sequences of the form
\begin{equation}\label{eq:u_n}
 u_0=u \mand u_{n+1}=f(u_n), \quad n=1,2,\ldots
\end{equation}

Clearly, each trajectory is eventually periodic. That is, there are some integers $0\leq s < t\leq q$ such that $u_s=u_t$ and thus $u_{s+n}=u_{t+n}$, $n=0,1,\ldots$
In particular, the smallest $T_u=t$ satisfying the above condition is called the {\it trajectory length}.

For $u\in \F_q$ let $\sP_u$ be the \textit{period length} of the sequence~\eqref{eq:u_n} with starting point $u$, that is, the minimal $k>0$ such that $u_{n}=u_{n+k}$ for $n= 0,1,\ldots$ with the convention, that $\sP_u=\infty$ if the sequence~\eqref{eq:u_n} is not periodic (just eventually periodic).

Here we consider the size $G(N)$ of the smallest subgroup $\cG\subseteq \F_q^*$ which contains the first $N$ non-zero elements of the sequence $(u_n)$, that is, 
$$
u_n\in\cG\cup \{0\}, \qquad n=1,\ldots, N.
$$

Some bounds on $G(N)$ and also on the frequency of the event $u_n \in \cG$ for a given subgroup $\cG\subseteq \F_q^*$, 
and related questions, are given in~\cite{Shp1, Shp2}.  However, the results of~\cite{Shp1} apply to either very long segments of an orbit 
or (as well as in~\cite{Shp2} are not uniform with respect to $f$ (that is, depend of the size of the coefficients of $f$). 
Here we use a different approach, utilizing a recent result of Vyugin~\cite{Vyu} and 
obtain  stronger  and fully uniform results. In fact, we study analogous problems in a much more general situation of semigroups generated by several polynomials under  composition.

Let $f_1,\ldots, f_k\in\F_q[X]$ be polynomials of positive degree. Consider the functional graph $\cH(\F_q)$
with vertices $\F_q$ and directed edges $(x,y)$ with $f_i(x)=y$, $1\leq i\leq k$.
 
For   a point $u\in\F_q$ we consider the vertices in this graph which are close to $u$:
$$
\cV_u(N)=\{f_{i_1} \circ\ldots \circ f_{i_n}(u):~ i_1,\ldots, i_n\in\{1,\ldots, k\},\  0\leq n\leq N\}
$$
and study whether these points are contained in small subgroups. Put
$$
 V_u(N)=\#\cV_u(N)
$$
and define $G_k(N)$ as the size of the smallest subgroup $\cG\subseteq \F_q^*$ which contains all non-zero elements of $\cV_u(N)$:
$$
\cV_u(N) \subseteq \cG \cup\{0\}.
$$

\subsection{Results} 
In order to state the results, for a graph $\cH(\F_q)$ we denote by $\overline{\cH}(\F_q)$  
the undirected graph (that is, $(x,y)$ is an edge in  $\overline{\cH}(\F_q)$ if $(x,y)$ or $(y,x)$ is an edge in $\cH(\F_q)$) and we also denote  by $\sC_0$ the length of a shortest cycle in  the graph $\overline{\cH}(\overline{\F}_q)$ (defined over the algebraic closure $\overline{\F}_q$ of $\overline{\F}_q$) 
which contains $0$. Clearly, for the case of one polynomial ($k=1$), we have $\sC_0=\sP_0$. 

We remark that the size (or finiteness) of  $\sC_0$   seems to   reappear in many works on unlikely intersections of orbits (in finite or infinite fields), 
see~\cite{NiWi1, OPS}, however there seems to be no intrinsic reason for this. It is highly desirable to 
gain better understanding of the phenomenon.  

Since these underlying results of~\cite{Vyu}  
apply only in prime fields, we limit our considerations to prime finite fields $\F_p$ where $p$ is a large prime. 
All explicit and implicit constants are independent of $p$.  

We now show that $G_k(N)$ grows at least quadratically compared to $V_u(N)$ as long as the field size $p$ is 
large enough to accommodate this growth, improving the trivial linear bound 
$$
G_k(N) \ge V_u(N).
$$

\begin{theorem}\label{thm:semigroup-1}
Let $\varepsilon>0$, $k\geq 1$ and $d\geq 2$ be fixed. Then there exist constants $c_1$ and $c_2$
 depending only on $d$, $k$ and  $\varepsilon$,  such that if $f_1,\ldots, f_k\in\F_p[X]$ are polynomials of degree at most $d$ and $\sC_0\geq c_1$,
 then
$$
G_k(N)\geq c_2\min\{V_u(N)^{2-\varepsilon},\,  p^{1-\varepsilon} \} \qquad \text{for } N\geq 1. 
$$
\end{theorem}

As a special case, we have the following result for the dynamical system~\eqref{eq:u_n} defined by just one polynomial.

\begin{cor} 
For any  $\varepsilon>0$ and $d\geq 2$ there exist constants $c_1$ and $c_2$
 depending only on $d$ and  $\varepsilon$ 
such that if $f\in\F_p[X]$ is of degree at most $d$,
$\sP_0 \geq c_1$ and $N<T_u$,
 where $T_u$ is the trajectory length of $(u_n)$,  then
$$
G(N) \geq c_2 \min\left\{N^{2-\varepsilon},\,  p^{1-\varepsilon}\right\}. 
$$
\end{cor}

We now show that for subgroups $\cG$ of size smaller than $G_k(N)$ the frequency of elements of the 
orbit $\cV_u(N)$ which fall in $\cG$ is small.
More precisely, for a subgroup $\cG\subseteq \F_p^*$ we denote
$$
T_k(N,\cG)=\#(\cV_u(N)\cap \cG )
$$
and next give an upper bound of the frequency 
$$
\rho_k(N) = \frac{T_k(N,\cG)}{V_u(N)}.
$$

\begin{theorem}\label{thm:semigroup-2}
For any $\varepsilon>0$, $d\geq 2$ and $k\geq 1$ there exist positive constants $c_1$ and $c_2$,
depending only on $k,d$ and $\varepsilon$, such that if $f_1,\dots, f_k\in \F_p[X]$ are of degree at most $d$, $\cG$ is a subgroup of $\F_p^*$ with
 $$
 V_u(N)^{\varepsilon}<\#\cG< \min\{ V_u(N)^{2-\varepsilon},\,  p^{1-\varepsilon}\},
 $$
 then for $2 \le V_u(N)<d^{c_1 \sC_0}$  we have
 $$
 \rho_k(N)<c_2 \frac{1}{\log V_u(N)}. 
 $$
\end{theorem}

As a special case, we have the following result for the dynamical system~\eqref{eq:u_n} defined by just one polynomial.

\begin{cor} 
For any $\varepsilon>0$ and $d\geq 2$ there exist positive constants $c_1$ and $c_2$,
depending only on $d$ and $\varepsilon$,   such that if $f\in \F_p[X]$ is of degree $d$, $\cG$ is a subgroup of $\F_p^*$ with
 $$
N^{\varepsilon}<\#\cG< \min\{N^{2-\varepsilon},\,  p^{1-\varepsilon}\},
 $$
 then for $2 \le N<\min\{d^{c_1 \sP_0},\, T_u\}$
 we have
 $$
 \rho(N)<c_2 \frac{1}{\log N}. 
 $$
\end{cor}

\section{Auxiliary results}

\subsection{Notation} 
We use the Landau symbol $O$ and the Vinogradov symbol $\ll$. Recall that the
assertions $U=O(V)$ and $U \ll V$  are both equivalent to the inequality $|U|\le cV$ with some  constant $c>0$,
which throughout this paper,  may depend on $k$ and $\varepsilon$, but is independent of  the prime 
$p$ and polynomials $f_1,\ldots, f_k\in\F_q[X]$.

For a finite  set $\cS \subseteq \Z$, it is convenient to define $\cS^\infty$ as the set of all finite 
words in the alphabet $\cS$.

\subsection{Dense sets in graphs} 

Here we present a graph theory result which can be of independent interest.

Let $\cH$ be a directed graph, with possible multiple edges. Let $\cV(\cH)$ be the set of vertices of $\cH$. For $u,v\in \cV(\cH)$, let $d(u,v)$ be the distance from $u$ to $v$, that is, the length of a shortest (directed) path from $u$ to $v$. Assume, that all the vertices have the out-degree $k\geq 1$, and the edges from all vertices are labeled by $\{1,\dots, k\}$. 

For a word $\omega\in \{1,\dots, k\}^\infty$ over the alphabet $\{1,\dots, k\}$ and $u\in \cV(\cH)$, let $\omega(u)\in \cV(\cH)$ be the end point of the walk started from $u$ and following the edges according to $\omega$.

Let us fix $u \in \cV(\cH)$ and a subset of vertices $\cA\subseteq \cV(\cH)$. Then for words $\omega_1,\dots, \omega_\ell$ put
\begin{align*}
L_N(u, \cA; \omega_1, \ldots, \omega_\ell) =\#\{v& \in \cV(\cH):~d(u,v)\leq N, \\ 
& d(u,\omega_i(v))\leq N, \ \omega_i(v) \in \cA,\ i=1,\ldots,\ell  \}.
\end{align*}

We need the following combinatorial statement which we believe can find further applications in  the
study of semigroups generated by several polynomials or other functions.

To state the results, for $k, h \geq 1$ let $B(k,h)$ denote the size of the complete $k$-tree of depth $h-1$, that is
\begin{equation}\label{eq:B}
 B(k,h)=
 \left\{
 \begin{array}{cl}
  h & \text{if } k=1,\\
  \frac{k^{h}-1}{k-1} & \text{otherwise}.
 \end{array}
 \right.
\end{equation}

\begin{lemma}
\label{lem:dense subgr} 
Let $u\in \cV(\cH)$, and $h, \ell\geq 1$ be fixed. If $\cA \subseteq \cV(\cH)$ is a subset of vertices with 
\begin{align*}
\#\{v\in \cA : ~d&(u,v)\leq N  \}\\
& \geq \max\left\{ 3B(k,h), \, \frac{3\ell}{h}\#\left\{v\in \cV(\cH): \  d(u,v)\leq N \right \} \right\},
\end{align*}
then there exist words $\omega_1, \ldots, \omega_\ell\in\{1,\dots, k\}^\infty$ of length at most $h$ such that
$$
L_N(u, \cA; \omega_1,\ldots, \omega_\ell)\gg  \frac{h}{B(k,h)^{\ell+1}}  \#\{v\in \cA: \  d(u,v)\leq N  \},
$$
where the implied constant depends only on $\ell$.
\end{lemma}

\begin{proof}
Put
$$
\beta=B(k,h) \mand \nu=\#\{v\in \cV(\cH) : \ d(u,v)\leq N  \}.
$$

Let $P$ be the number of pairs $(z, v)\in \cV(\cH) \times \cA$ such that
$$
d(u,z)\leq N-h \quad d(z,v)\leq h.
$$

If $v\in\cA$ and $h \leq  d(u,v)\leq N$, then there is a path  from $u$ to $v$ of length at least $h$, thus there are at least $h$ vertices $z\in \cV(\cH)$ such that $d(u,z)\leq N-h$ and $d(z,v)\leq h$. As there are at most $\beta$ vertices $v$ with $d(u,v)<h$, we have
\begin{equation}\label{eq:P-lower}
P  \geq h(\#\{v\in \cA : ~d(u,v)\leq N  \} - \beta )
\end{equation}
On the other hand, let $J$ be the number of $z\in \cV(\cH)$ such that $d(u,z)\leq N-h$ and 
$$
\#\{v\in \cA: d(z,v)\leq h
\}\geq \ell.
$$ 
Then 
\begin{equation}\label{eq:P-upper}
P\leq \beta J+(\ell-1)(\nu-J ).
\end{equation}

Comparing~\eqref{eq:P-lower} and~\eqref{eq:P-upper}, we derive
\begin{align*}
J&\gg \frac{h \#\{v\in \cA : \ d(u,v)\leq N  \}  -(\ell-1) \nu- h\beta}{\beta-\ell+1}\\
&\gg  \frac{h }{\beta}\#\{v\in \cA : \ d(u,v)\leq N  \}.
\end{align*}
There are at most 
$$
\binom{\beta}{\ell} \le \beta^\ell
$$ 
choices for $\omega_1, \ldots, \omega_\ell$, so we see that there is at least one choice, that
\begin{align*}
L_N(u, \cA; \omega_1,\ldots, \omega_k)&\gg  \frac{h }{\beta\binom{\beta}{\ell}}\#\{v\in \cA : \ d(u,v)\leq N  \}\\
&\gg  \frac{h }{\beta^{\ell+1}}\#\{v\in \cA : \ d(u,v)\leq N  \}
\end{align*}
which concludes the proof. 
\end{proof}

\subsection{Polynomial values of subgroups}

We call the set of polynomials $g_1,\ldots, g_k\in\F_p[X]$ \textit{admissible} if there exist $x_1,\ldots, x_k\in \overline{\F}_p$ 
such that
$$
g_i(x_i)=0 \quad \text{but} \quad g_j(x_i)\neq 0 \quad \text{for} \quad 0\leq i,j \leq k, i\neq j.
$$

\begin{lemma}\label{lemma:adm-semigroup}
 For  polynomials $g_1,\ldots, g_{\ell}\in\F_p[X]$, if $h<\sC_0/2$, then the polynomials
 \begin{equation*}
 g_{i_1}\circ \ldots \circ g_{i_r},\quad r\leq h,\ i_1,\ldots, i_r\in\{1,\ldots, \ell\}
 \end{equation*}
 are admissible.
\end{lemma}
\begin{proof}
For a given $F$ of form 
$F= g_{i_1}\circ \ldots \circ g_{i_r}$, $r\leq h$, $i_1,\ldots, i_r\in\{1,\ldots, \ell\}$, let $x\in\overline{\F}_p$ be a zero of $F$. If it is also a zero of $G$ of form $G= g_{j_1}\circ \ldots \circ g_{j_s}$, $s\leq h$, $j_1,\ldots, j_r\in\{1,\ldots, \ell\}$, $G\neq F$, then there are two different paths from $x$ to 0 in the functional graph $\cH(\overline{\F}_p)$ along the edges $i_r, \dots, i_1$ and $j_s,\dots, j_1$. As $r,s\leq h$ these two paths define a cycle of size at most $r+s\leq 2s< \sC_0$, a contradiction. 
\end{proof}
 
Our main tool is the following result of~Vyugin~\cite{Vyu}. 
 
\begin{lemma}\label{lemma:Vyugin}
 Let $\cG$ be a subgroup of $\F_p^*$   and let $\cG_1,\ldots, \cG_k$ be  $k \ge 2$ cosets of $\cG$.
 If $g_1,\ldots, g_k\in\F_p[X]$ form an admissible set of polynomials of degrees $m_1\leq\ldots \leq m_k$ and
 $$
 C_1(\mathbf{m},k)< \# \cG < C_2(\mathbf{m},k)p^{1-\frac{1}{2k+1}},
 $$
 then
 $$
 \#\{x:~g_i(x)\in \cG_i, \ i=1,\ldots, k \}\leq C_3(\mathbf{m},k) \#\cG^{\frac{1}{2}+\frac{1}{2k}},
 $$
 where for $\mathbf{m} = (m_1, \ldots, m_k)$ we define
 $$
 C_1(\mathbf{m},k)=2^{2k}m_k^{4k}, \qquad C_2(\mathbf{m},k)=(k+1)^{-\frac{2k}{2k+1}}(m_1\cdots m_k)^{-\frac{2}{2k+1}},
 $$
 and 
 $$
 C_3(\mathbf{m},k)=4(k+1)(m_1+\ldots+m_k)(m_1\cdots m_k)^{\frac{1}{k}}. 
 $$
\end{lemma}

\section{Proof of main results}

\subsection{Proof of Theorem~\ref{thm:semigroup-1}}

Let $h \geq 1$ be the smallest integer such that
$$
\frac{1}{2B(k,h)}<\varepsilon
$$
where $B(k,h)$ is defined by~\eqref{eq:B}.  For an appropriate choice of the constant $c_1$, we can assume, that $h<\sC_0/2$.

We also set
$$
\beta = B(k,h).
$$

Let $\cG$ be the group generated by $\cV_u(N)$. Then
$$
f_{i_1}\circ \ldots \circ f_{i_r}(x)\in\cG,
$$
for 
$$
x\in \cV_u(N-h), \ i_1,\ldots, i_r\in\{1,\ldots,k\},\ 1\leq r \leq h. 
$$
Moreover, the polynomials 
$$
f_{i_1}\circ \ldots \circ f_{i_r},\quad i_1,\ldots, i_r\in\{1,\ldots,k\},\ 1\leq r \leq h, 
$$
are admissible by Lemma~\ref{lemma:adm-semigroup}.

For an appropriate choice of the constant $c_1$ we have
$$
\#\cG  > 2^{2\beta}d^{4h \beta}. 
$$
We can also  assume, that
$$
\#\cG< \((\beta+1)^{-1}d^{-1}p\) ^{2\beta/(2\beta +1)}.
$$
as otherwise there is nothing to prove.
Then by  Lemma~\ref{lemma:Vyugin}
$$
V_u(N-h)\ll \#\cG^{(\beta+1)/(2\beta +1)}.
$$
Finally, we remark that
$$
V_u(N)\leq \beta V_u(N-h)
$$
which gives the desired result.

\subsection{Proof of Theorem~\ref{thm:semigroup-2}}

We can assume throughout the proof, that both $p$ and $N$ are large enough.

Define
$$
\ell=\left\lceil \frac{4}{\varepsilon}-1 \right\rceil \mand 
h=\left\lfloor \frac{1 }{2\ell} \frac{\log \#\cG}{(\ell+1) \log k+2\log d}\right\rfloor, 
$$
Put
$$
\delta_0=\frac{1+\log k}{2\ell((\ell+1)\log k+2\log d)}.
$$
Clearly, $0<\delta_0<1/2$. By the above choice of parameters, we have
$$
B(k,h) \leq V_u(N)^{(2-\varepsilon) \delta_0}.
$$
Assume, that
\begin{equation}\label{eq:assump}
\rho_k(N)\geq  \frac{6\ell(1+\log k)}{\varepsilon\delta_0 \log V_u(N)}.
\end{equation}
Then we have
$$
T_k(N,\cG)\geq \max\left\{3B(k,h),\,  \frac{3\ell}{h}V_u(N) \right\}. 
$$

By Lemma~\ref{lem:dense subgr}, there are $\omega_1, \ldots, \omega_\ell\in\{1,\dots, k\}^\infty$ of form
$\omega_i=(\omega_{i,1},\ldots, \omega_{i,r_i})$, $r_i\leq h$, $1\leq i\leq \ell$
such that
\begin{multline}
\label{eq:lower-semigroup}
\#\{v\in \cV_u (N-h):~f_{\omega_{i,1}}\circ \ldots \circ  f_{\omega_{i,r_i}}(v)\in \cG ,\  i=1, \ldots, \ell \}
\\
\gg  \frac{h V_u(N)}{B(k,h)^{\ell+1}} \rho_k(N).
\end{multline}

Consider the polynomials $F_i=f_{\omega_{i,1}}\circ \ldots \circ f_{\omega_{i,r_i}}$ of degree $m_i=\deg F_i\leq d^h$, $i=1,\ldots, \ell$.

If $\cG$ is large enough in terms of $\ell$, then 
$$
C_1(\mathbf{m}, \ell)\leq 2^{2\ell}d^{4\ell h}\leq \# \cG.
$$
Moreover, if $\ell$ is large enough, then
\begin{align*}
C_2(\mathbf{m}, \ell) p^{\frac{2\ell}{2\ell+1}}&\geq (\ell+1)^{-\frac{2\ell}{2\ell+1}} d^{-\frac{2\ell}{2\ell+1}h}p^{\frac{2\ell}{2\ell+1}}\\
&\geq
(\ell+1)^{-\frac{2\ell}{2\ell+1}} d^{-\frac{2\ell}{2\ell+1}h}\#\cG^{\frac{1}{1-\varepsilon}\frac{2\ell}{2\ell+1}}
\geq \#\cG.
\end{align*}

By the choice of $h$,
$$
h \leq \frac{1 }{2\ell} \frac{(2-\varepsilon)\log V_u(N)}{(\ell+1) \log k+2\log d}  < \frac{1}{2} \sC_0,
$$
by an appropriate choice of $c_1$. Then the polynomials $F_1,\ldots, F_\ell$ are admissible by Lemma~\ref{lemma:adm-semigroup}. Then by Lemma~\ref{lemma:Vyugin} we have
\begin{equation}
\begin{split} 
\label{eq:upper-semigroup}
\#\{
v&\in \cV_u (N-h):~F_i(v)\in \cG ,\  i=1, \ldots, \ell \} \\ 
&\quad \leq \#\{x\in\F_p:~F_i(x)\in \cG ,\ i=1, \ldots, \ell \}\ll d^{2h} \#\cG^{\frac{1}{2}+\frac{1}{2\ell}}.
\end{split} 
\end{equation}
Comparing~\eqref{eq:lower-semigroup} and~\eqref{eq:upper-semigroup}, we have
\begin{align*}
\rho_k(N)&\ll \frac{B(k,h)^{\ell+1}d^{2h} \#\cG^{\frac{1}{2}+\frac{1}{2\ell}}}{h V_u(N)}\ll  \frac{h^\ell\#\cG^{\frac{1}{2}+\frac{1}{\ell}}}{ V_u(N)}\\
&\ll \left(\frac{(2-\varepsilon)\log V_u(N)}{4\ell \log d} \right)^\ell V_u(N)^{(2-\varepsilon)(\frac{1}{2}+\frac{1}{\ell} )-1 }.
\end{align*}
Using $\ell <  4\varepsilon^{-1} -2$ we see then the exponent of $V_u(N)$  is negative, that is
$$
(2-\varepsilon)\(\frac{1}{2}+\frac{1}{\ell}\)-1 = \frac{2}{\ell} - \frac{\varepsilon}{2} -  \frac{\varepsilon}{\ell} 
=    -  \frac{\varepsilon}{2\ell} \(\ell + 2 - 4\varepsilon^{-1}\)  <0, 
$$
which contradicts~\eqref{eq:assump}, provided that  $V_u(N)$ is large enough. 

\section*{Acknowledgement}

The research of L.M. was supported by the Austrian Science Fund (FWF): Project  P31762, and  of I.S. was  
supported in part by the Australian Research Council Grants DP170100786 and DP180100201.

\end{document}